\documentclass[10pt,a4paper,reqno]{article}
\usepackage{amssymb,amsmath,amsthm}
\usepackage[T1]{fontenc}
\usepackage[dvips]{graphicx}

\theoremstyle{plain}
\newtheorem{theorem}{Theorem}
\newtheorem{corollary}{Corollary}

\newtheorem*{theorem A}{Theorem A}
\newtheorem*{theorem B}{Theorem B}
\newtheorem*{corollary A}{Corollary A}
\newtheorem*{corollary B}{Corollary B}
\theoremstyle{definition}

\newtheorem{rem}{Remark}
\newtheorem{example}{Example}
\newtheorem*{proof A}{Proof of Theorem A}
\newtheorem*{proof B}{Proof of Theorem B}
\newtheorem*{proof AA}{Proof of Corollary A}

\DeclareMathOperator*{\p}{\bf P}
\DeclareMathOperator*{\e}{\bf E}
\begin{document}

\title{A Borel-Cantelli lemma and its applications}
\maketitle \centerline{\small{NUNO LUZIA}}
\centerline{\emph{Universidade Federal do Rio de Janeiro, Brazil}}
\centerline{\emph{e-mail address}: nuno@im.ufrj.br} \vspace{24pt}
\centerline{\scriptsize ABSTRACT}
\text{}\\
\begin{footnotesize}
We give a version of the Borel-Cantelli lemma. 
As an application, we prove an \emph{almost sure local central limit theorem}. As another application, we prove a dynamical Borel-Cantelli lemma for systems with sufficiently fast decay of correlations with respect to Lipschitz observables. 
\end{footnotesize}
\text{}
\text{}\\\\
\tableofcontents
\section{Introduction and statements}
The classical Borel-Cantelli lemmas are a powerful tool in Probability Theory and Dynamical Systems. Let $(\Omega, \mathcal{F}, \p)$ be a probability space and $(A_n)$ a sequence of measurable sets in $\mathcal{F}$. These lemmas say that (see \cite{7} for proofs): 
\begin{enumerate}
\item[(BC1)] If $\sum_{n=1}^\infty \p(A_n)<\infty$ then $\p(x\in A_n \text{ i.o.})=0$.
\item[(BC2)] If the sets $A_n$ are independent and $\sum_{n=1}^\infty \p(A_n)=\infty$ then $\p(x\in A_n \text{ i.o.})=1$.
\item[(BC3)] If the sets $A_n$ are pairwise independent and $\sum_{n=1}^\infty \p(A_n)=\infty$ then
 \[
    \frac{\sum_{i=1}^{n} 1_{A_i}}{\sum_{i=1}^{n} \p(A_i)} \to 1 \quad \text{a.s.}
 \]
\end{enumerate}
Here $1_{A_i}$ is the indicator function of the set $A_i$. Note that (BC3) implies (BC2), but the proof of (BC3) is more elaborated. 

\subsection{A Borel-Cantelli Lemma}

\begin{theorem}\label{cheb}
 Let $X_i$ be non-negative random variables and $S_n=\sum_{i=1}^{n} X_i$. If $\sup {\bf E}X_i< \infty$, ${\e}S_n\to\infty$ and there exists $\gamma>1$ such that
 \begin{equation}\label{rest}
      \mathrm{var}(S_n)= O\left(\frac{({\e}S_n)^2}{(\log {\e}S_n) (\log\log {\e}S_n)^\gamma}\right)
 \end{equation}
then      
\[      
      \frac{S_n}{{\e}S_n}\to 1  \quad \text{a.s.}
\]
\end{theorem}

We see that Theorem \ref{cheb} implies (BC3), because when $A_i$ are pairwise independent sets, $X_i=1_{A_i}$ and ${\e}S_n\to\infty$ then $\mathrm{var}(S_n)=\sum_{i=1}^{n} \mathrm{var}(X_i)\le{\e}S_n$. 
A slight modification of Theorem \ref{cheb} also gives a version of the Strong Law of Large Numbers without assuming the random variables are pairwise independent.

\begin{corollary}\label{slln}
 Let $X_i$ be identically distributed random variables with ${\e}X_i=\mu$, ${\e}X_i^2<\infty$ and $S_n=\sum_{i=1}^{n} X_i$. If $X_i\ge -M$, for some constant $M>0$, and there exists $\gamma>1$ such that
\[
   \sum_{1\le i<j\le n} \Bigl({\e}(X_iX_j)-\mu^2\Bigr)= O\left(\frac{n^2}{(\log n) (\log\log n)^{\gamma}}\right).
\]
then      
\[      
      \frac{S_n}{n}\to \mu  \quad \text{a.s.}
\]
\end{corollary}

\subsection{An almost sure local central limit theorem} 
Let $X_i$ be independent random variables such that each $X_i$ assume the values $+1$ and $-1$ with probabilities $1/2$ and $1/2$. Then $S_n=\sum_{i=1}^n X_i$ is the simple random walk on the line. 
It is well known that the sequence of random variables $1_{\{S_{i}=0\}}$ does not obey the law of large numbers. More precisely (see \cite{12}),
\[
    \limsup_{n\to\infty} \frac{1}{\sqrt{n \log \log n}} \sum_{i=1}^n  1_{\{S_{i}=0\}} = \sqrt{2} \quad\text{a.s.}
\]
and there exists a constant $0<\gamma_0<\infty$ such that
\[
    \liminf_{n\to\infty} \frac{\sqrt{\log \log n}}{\sqrt{n}} \sum_{i=1}^n  1_{\{S_{i}=0\}} = \gamma_0 \quad\text{a.s.} 
\]
It is then natural to ask if $1_{\{S_{n_i}=0\}}$ obeys the law of large numbers for some increasing sequence $n_i$ of even positive integers.

More generally, we consider i.i.d. random variables $X_i$ which are \emph{h-lattice valued}, i.e. $\sum_{k\in\mathbf{Z}} \p(X_i=kh+b)=1$, for some $h>0$ and $b\in\mathbf{R}$ (we assume $h$ with this property is maximal).
Let $S_n=\sum_{i=1}^n X_i$ and $a\in\mathbf{R}$. By abuse of notation, when we write $S_{n}=a\sqrt{n}$ we mean  $S_{n}=[(a\sqrt{n}-nb)/h]h+nb$.

\begin{theorem}\label{return1}
Let $X_i$ be i.i.d. h-lattice valued random variables with ${\e}X_i=0$, ${\e}X_i^2=\sigma^2>0$ and ${\e}|X_i|^3<\infty$, and $S_n=\sum_{i=1}^n X_i$. Let $n_i$ be an increasing sequence of positive integers and $a\in\mathbf{R}$. Then
 
\begin{enumerate}
\item[(a)] If  $\sum_{i=1}^\infty n_i^{-1/2}<\infty$ then $\p(S_{n_i}=a\sqrt{n_i} \text{ i.o.})=0$.\\
\item[(b)] If  there exist $A>0$ and $\gamma>1$ such that
\[ 
  n_{i+1}-n_i\ge A n_i^{1/2} \quad\text{and}\quad
  n_i\le A^{-1} i^2 (\log i) (\log \log i)^{-3} (\log \log \log i)^{-2\gamma} 
\]
for all $i$, then
\begin{equation}\label{quot}
 \frac{\sum_{i=1}^n 1_{\{S_{n_i}=a\sigma\sqrt{n_i}\}}}{\sum_{i=1}^n n_i^{-1/2}} \to \frac{h}{\sqrt{2\pi} \sigma}\, e^{-a^2/2} \quad\text{a.s.}
 \end{equation}
 
Let $\Delta_n^a$ be the quotient between the left and right hand sides of (\ref{quot}). Then, for every $N>0$ there exists $C>0$ such that, for every $\epsilon>0$,
\[
\sup_{a\in [-N,N]} \p\left(|\Delta_n^a-1|>\epsilon\right)\le C\epsilon^{-2}(\log \log n)^{-1}(\log\log\log n)^{-\gamma}.
\]
\end{enumerate}
\end{theorem}

\begin{rem}
Concerning the divergent case $\sum_{i=1}^\infty n_i^{-1/2}=\infty$ and $a=0$. In \cite{3} the authors prove that  if there exists an integer $m>0$ such that $n_{i+mj}\ge n_{i}+n_{j}$ for every $i,j$ then $\p(S_{n_i}=0 \text{ i.o.})=1$. In \cite{5} the same authors claim that if there exist $A>0$ such that $n_{i+1}-n_{i}>A n_{i}^{1/2}$ for every $i$ then $\p(S_{n_i}=0 \text{ i.o.})=1$, however it seems their proof is not correct (see p. 184, Equation (21)).
\end{rem}
Considering $n_i=i^2$ and the `change of variable' $k=i^2$ we get the following \emph{almost sure local central limit theorem}.
\begin{corollary}
With the same hypotheses of Theorem \ref{return1},
\[
       \frac{1}{\log n} \sum_{k=1}^{n} 1_{\{S_{k}=a\sigma\sqrt{k}\}} \frac{1}{\sqrt{k}}\to \frac{h}{\sqrt{2\pi}\sigma} \, e^{-a^2/2}\quad \text{a.s.} 
\]
\end{corollary}
When $a=0$ this was proved in \cite{4} (see also \cite{6}). For $a\in\mathbf{R}$ this was proved, independently, in \cite{13}.

The version for random variables with density follows. 
\begin{theorem}\label{return3}
Let $X_i$ be i.i.d. random variables having density function whose Fourier tranform (or some positive integer power of it) is integrable, ${\e}X_i=0$, ${\e}X_i^2=\sigma^2>0$ and ${\e}|X_i|^3<\infty$, and $S_n=\sum_{i=1}^n X_i$. Let $a\in\mathbf{R}$ and $n_i$ be an increasing sequence of positive integers satisfying 
\begin{equation*}\label{seqexp}
      \frac{n_{i+1}}{n_i}\ge1+\frac{A(\log i)(\log\log i)^{\alpha}}{i}\quad \text{for all }i 
\end{equation*}
for some $A>0$ and $\alpha>2$. Then
\begin{equation}\label{densidade}
\frac{1}{n}\sum_{i=1}^n 1_{\{S_{n_i}=a\sigma\sqrt{n_i}\}} \to  \frac{1}{\sqrt{2\pi}}\, e^{-a^2/2} \quad\text{a.s.}
\end{equation}
For the special case $a=0$, (\ref{densidade}) holds with $n_i=i$. 

Let $\Delta_n^a$ be the quotient between the left and right hand sides of (\ref{densidade}). Then, for every $N>0$ there exists $C>0$ such that, for every $\epsilon>0$,
\[
\sup_{a\in [-N,N]} \p\left(|\Delta_n^a-1|>\epsilon\right)\le C \epsilon^{-2}(\log n)^{-1}(\log\log n)^{1-\alpha}.
\]
\end{theorem}

In above theorem we can consider sequences of type $n_i=[e^{A(\log i)^{2}(\log \log i)^\alpha}]$. In particular, considering the sequence
$n_i=2^i$ and the `change of variable' $k=2^i$ we get the following \emph{almost sure local central limit theorem}.

\begin{corollary}
With the same hypotheses of Theorem \ref{return3},
\[
      \frac{1}{\log n} \sum_{k=1}^{n} 1_{\{S_{k}=a\sigma\sqrt{k}\}} \frac{1}{{k}}\to \frac{1}{\sqrt{2\pi}} e^{-a^2/2} \quad\text{a.s.}
\]
\end{corollary}

\begin{rem}
Using the same techniques (and the Berry-Esseen Theorem) we can prove Theorem 3 with `$\le$' instead of `$=$' in (\ref{densidade}), for i.i.d. random variables $X_i$ with finite third moment, thus giving  a new proof of the \emph{almost sure central limit theorem} (for related results, see \cite{11} and references therein).
\end{rem}

\subsection{A dynamical Borel-Cantelli Lemma}
We want to consider the dynamical version of (BC3). Let $(X,d)$ be a metric space, $\mu$ be a Borel probability measure on $X$, and $T$ be a $\mu$-preserving transformation on $X$. Let $(B_n)$ be a sequence of measurable sets in $X$. In what conditions does
\[
 \text{(DBC)}\quad\quad\quad\quad \frac{\sum_{i=0}^{n-1} 1_{B_i}(T^ix)}{\sum_{i=0}^{n-1} \mu(B_i)} \to 1 
    \quad \text{for $\mu$ a.e. } x 
\]      
holds? We can easily find sufficient conditions for (DBC) to hold: 
\begin{itemize}
\item The sets $B_n$ are all equal to $B$, $\mu(B)>0$ and $\mu$ is ergodic.
\item  The sets $T^{-n}B_n$ are pairwise independent and $\sum \mu(B_n)=\infty$. 
\end{itemize}
The first one follows from Birkhoff ergodic theorem, and the second one by (BC3) since $1_{B_i}(T^ix)=1_{T^{-i}B_i}(x)$.  However, it is very unlikely for a dynamical system $(T,X,\mu)$ to have $T^{-n}B_n$ pairwise independent. A far more reasonable condition is to have some sufficiently fast decay of correlations.   
In this section we extend some results of \cite{10} where the same kind of problem is treated. For related results we refer the reader to \cite{10} and references therein.

We denote by $\|  \cdot \|_\mathrm{Lip}$ the usual Lipschitz norm.
We say that $(T, X, \mu, d)$ has \emph{polynomial decay of correlations} (for Lipschitz observables) if, for every Lipschitz functions $\varphi,\psi\colon X\to\mathbb{R}$, 
\begin{equation}\label{cor}
\left|  \int \varphi\circ T^n \,\psi\,d\mu - \int \varphi\,d\mu  \int \psi\,d\mu \right| \le c(n)  \|  \varphi \|_\mathrm{Lip}  \|  \psi \|_\mathrm{Lip} 
\end{equation}
where $c(n)\le C n^{-\alpha}$ for some constants $C>0$ and $\alpha>0$ (rate). 

We say that $(T, X, \mu, d)$ has \emph{$\beta$-exponential decay of correlations}, $\beta>0$, if  $c(n)\le C e^{-\alpha n^\beta}$, for some constants $C,\alpha>0$, in (\ref{cor}). For $\beta=1$ we get the usual definition of  exponential decay of correlations.
For $\beta<1$ this is also known as \emph{streched exponential decay of correlations}.

We will assume the following condition on the measure $\mu$. There exist $C>0$, $\delta>0$  ($\delta<2$) and $r_0>0$ such that for all $x_0\in X$, $0<r\le r_0$ and $0<\epsilon\le 1$,
\begin{equation*}
 \text{(A)}\quad\quad\quad  \mu(\{x : r< d(x,x_0)< r+\epsilon\}) < C \epsilon^\delta.
\end{equation*}
In what follows, if  we only consider nested balls $B_i=B(x_0,r_i)$ ($r_i\to 0$) centered at a given point $x_0$ then we only require (A) holds for the point $x_0$, in other words, $r\mapsto \mu(B(x_0,r))$ is H\"older continuous (with exponent $\delta(x_0)>0$).

\begin{example}\text{ }
\begin{enumerate}
\item
If $X$ is a compact manifold and $\mu$ is absolutely continuous with respect to Lebesgue measure with density in $L^{1+\alpha}$ for some $\alpha>0$, then $\mu$ satisfies (A). 
\item
If $X=[0,1]^2$ and  $\mu=\nu \times\mu_x$ where $\mu_x$  are absolutely continuous with respect to Lebesgue measure with densities uniformly bounded in $L^\infty$, then there exists $C>0$ such that the $\mu$ measure of any annulus of inner radius $r$ and width $\epsilon$ is bounded by $C \sqrt{\epsilon}$, and so $\mu$ satisfies (A). 
\item
If $X=[0,1]$ and $\mu$ is the usual measure supported on the middle-third Cantor set, then $F(x)=\mu([0,x])$ is the Devil's staircase which is H\"older continuous with exponent $\log 2/\log 3$, and so  $\mu$ satisfies (A).
\end{enumerate}
\end{example}

\begin{theorem}\label{teorprinc}
Suppose $\mu$ satisfies (A) and $B_i$ are balls such that $\sum_{i=0}^{n-1} \mu(B_i)\ge n^\beta (\log n)^\gamma$, for some $0<\beta<1$ and $\gamma>1+\frac{(2-\delta)\beta}{2+\delta}$. If $(T, X, \mu, d)$ has polynomial decay of correlations with rate $\alpha=  \frac{2/\delta+1}{\beta}-1$
then (DBC) is satisfyed.

Let $\Delta_n$ be the left hand side of (DBC). Then, for every $\epsilon>0$, $\rho>1$ and $0<\theta<\frac{\gamma(2+\delta)-1}{2\beta}+\frac{\delta}{2}-1$, there exist $C>0$ and $C_\theta>0$ ($C_\theta$ does not depend on $\epsilon,\rho$) such that
\[
 \mu( |\Delta_n -1|>\epsilon )\le C (\log n)^{-1}(\log \log n)^{-\rho}+C_\theta (\log n)^{-\theta}.
\]
\end{theorem}

\begin{corollary}\label{corol1}
Suppose $\mu$ satisfies (A) and $B_i$ are balls such that $\mu(B_i)\ge i^{-\beta} (\log i)^\gamma$, for some $0<\beta<1$ and $\gamma>1+\frac{(2-\delta)(1-\beta)}{2+\delta}$. If $(T, X, \mu, d)$ has polynomial decay of correlations with rate $\alpha=  \frac{2/\delta+\beta}{1-\beta}$ then (DBC) is satisfyed.
\end{corollary}
Corollary \ref{corol1} is stronger than Theorem 3.1 of \cite{10}.

\begin{theorem}\label{teorsec}
Suppose $\mu$ satisfies (A) and $B_i$ are balls such that $\sum_{i=0}^{n-1} \mu(B_i)\ge  (\log n)^\beta (\log\log n) (\log\log\log n)^\gamma$, for some $\beta>0$, $\gamma>1$. If $(T, X, \mu, d)$ has $\beta^{-1}$-exponential decay of correlations then (DBC) is satisfyed.

Let $\Delta_n$ be the left hand side of (DBC). Then, for every $\epsilon>0$, there exists $C>0$ such that
\[
 \mu( |\Delta_n -1|>\epsilon )\le C (\log \log n)^{-1}( \log \log \log n)^{-\gamma}.
\]
\end{theorem}

\begin{corollary}\label{cor2}
Suppose $\mu$ satisfies (A) and $B_i$ are balls such that $\mu(B_i)\ge i^{-1} (\log i)^{\beta-1} (\log\log i) (\log\log\log i)^\gamma$, for some $\beta>0$, $\gamma>1$. If   $(T, X, \mu, d)$ has $\beta^{-1}$-exponential decay of correlations then (DBC) is satisfyed.
\end{corollary}
Corollary \ref{cor2} with $\beta=1$ is stronger than Theorem 4.1 of \cite{10}.

Many `nonuniformly hyperbolic dynamical systems' exhibit some kind of decay of correlations, see \cite{10} for examples  with polynomial and exponential decay of correlations. The `Viana-like maps'  are a class of dynamical systems which possess streched exponential decay of correlations (but no exponential decay of correlations),
see \cite{1}, so our results also give good recurrence results for these systems.
\subsection{Quantitative recurrence}

As before, let $(T, X, \mu)$ be a measure preserving transformation of a probability space $X$ which is also endowed with a metric $d$. For $\alpha>0$, we denote by $\mathcal{H}^\alpha$ the $\alpha$-Hausdorff measure of $(X, d)$. One of the most beautiful results on the recurrence of dynamical systems is the following.

\begin{theorem}\emph{(Boshernitzan \cite{2})}
If, for some $\alpha>0$, $\mathcal{H}^\alpha$ is $\sigma$-finite on $X$, then
\[
      \liminf_{n\to\infty} n^{\frac{1}{\alpha}}\, d(T^n(x),x)<\infty \quad \text{for $\mu$ a.e. } x \in X.
\]
If, moreover, $\mathcal{H}^\alpha(X)=0$, then
\[
      \liminf_{n\to\infty} n^{\frac{1}{\alpha}}\, d(T^n(x),x)=0 \quad \text{for $\mu$ a.e. } x \in X.
\]
\end{theorem}

Also a very nice result in this direction is as follows. Given $x_0\in X$, let $\bar{d}_\mu(x_0)$ be the \emph{upper pointwise dimension of $\mu$} at $x_0$ defined by
\[
  \bar{d}_\mu(x_0) =\limsup_{r\to0} \frac{\log \mu(B(x_0, r))}{\log r}       
\]
where $B(x_0, r)$ is the ball centered at $x_0$ of radius $r$. We also say that $(T, X, \mu, d)$ has \emph{superpolynomial decay of correlations} (for Lipschitz observables) if it has polynomial decay of correlations with rate $\alpha$ for every $\alpha>0$. 

\begin{theorem}\notag\emph{(Galatolo \cite{9})} If $(T, X, \mu, d)$ has superpolynomial decay of correlations, then, for every $x_0\in X$ and $\alpha> \bar{d}_\mu(x_0)$,
\[
      \liminf_{n\to\infty} n^{\frac{1}{\alpha}}\, d(T^n(x),x_0)=0\quad \text{for $\mu$ a.e. } x \in X.
\]  
\end{theorem}

The dynamical Borel-Cantelli lemmas (for nested balls) stated in previous section give, under additional assumptions, quantitative versions of these results. Let $\underline{\Theta}^\alpha_\mu(x_0)$ be the  \emph{lower $\alpha$-density of $\mu$} at $x_0$ defined by
\begin{equation*}\label{density}
      \underline{\Theta}^\alpha_\mu(x_0)=\liminf_{r\to 0} \frac{\mu(B(x_0,r))}{r^\alpha}.
\end{equation*}
So  $0\le\underline{\Theta}^\alpha_\mu(x_0)\le\infty$. 

\begin{theorem}\label{quant}\text{ }
\begin{enumerate}
\item[(a)]
If $(T, X, \mu, d)$ has polynomial decay of correlations with rate $\vartheta>0$, then, for every $x_0\in X$ satisfying (A) with 
$\delta(x_0)>2/\vartheta$ and $\mu(\{x_0\})=0$,  $\beta=\frac{\vartheta -2/\delta(x_0)}{\vartheta+1}$, $\gamma>1+\frac{(2-\delta(x_0))(1-\beta)}{2+\delta(x_0)}$, there exists $\kappa(n)$ with 
\begin{equation}\label{kapa}
      \limsup_{n\to\infty}\frac{\log \kappa(n)}{\log n}\le0
\end{equation}
such that, with $\alpha=\bar{d}_\mu(x_0)$,
\begin{equation}\label{quantpol}
    \lim_{N\to\infty} \frac{ \# \left\{1\le n \le N : \; n^{\frac{\beta}{\alpha}}\, d(T^n(x),x_0)< \kappa(n) \right\} } { N^{1-\beta} (\log N)^\gamma}= \theta(1-\beta)^{-1}
\end{equation}
($\theta=1$) for $\mu$ a.e.  $x \in X$.

If, moreover, $\underline{\Theta}^\alpha_\mu(x_0)>0$, then (\ref{quantpol}) holds with $\kappa(n)=(\log  i)^{\gamma/ \alpha}$ and 
$\theta=\underline{\Theta}^\alpha_\mu(x_0)$.

In particular, if  $(T, X, \mu, d)$ has superpolynomial decay of correlations, then, for every $x_0\in X$ satisfying (A) and $\mu(\{x_0\})=0$, $\beta<1$,   $\gamma>1+\frac{(2-\delta(x_0))(1-\beta)}{2+\delta(x_0)}$, (\ref{quantpol}) holds.
\item[(b)]
If $(T, X, \mu, d)$ has $\beta$-exponential decay of correlations, then, for every $x_0\in X$ satisfying (A) and $\mu(\{x_0\})=0$, $\gamma>1$, there exists $\kappa(n)$ 
satisfying (\ref{kapa}) such that, with $\alpha=\bar{d}_\mu(x_0)$,
\begin{equation}\label{quantexp}
    \lim_{N\to\infty} \frac{\# \left\{1\le n \le N : \; n^{\frac{1}{\alpha}}  \, d(T^n(x),x_0)< \kappa(n) \right\}}
    {(\log N)^{\beta^{-1}} (\log \log N)^\gamma} =\theta\beta
\end{equation}
($\theta=1$) for $\mu$ a.e.  $x \in X$.

If, moreover, $\underline{\Theta}^\alpha_\mu(x_0)>0$, then (\ref{quantexp}) holds with 
\[
   \kappa(n)=(\log i)^{(\beta^{-1}-1)/ \alpha} (\log \log i)^{\gamma/ \alpha}
\] 
and $\theta=\underline{\Theta}^\alpha_\mu(x_0)$.
\end{enumerate}
\end{theorem}

\section{Proofs}
\subsection{Proofs of 1.1} 

\begin{proof}[Proof of Theorem \ref{cheb}]
Given $\sigma>0$, Chebyshev's inequality implies
\[
      \p(|S_n - {\e}S_n|>\sigma{\e}S_n)\le \mathrm{var}(S_n)/(\sigma {\e}S_n)^2
\]
so
\begin{equation}\label{lemma}
\p(|S_n-{\e}S_n|>\sigma{\e}S_n)\le C \sigma^{-2} (\log {\e}S_n)^{-1}(\log\log {\e}S_n)^{-\gamma}
\end{equation}
for some constant $C>0$.
Note that (\ref{lemma}) implies  $S_n/{\e}S_n\to 1$ in probability. To get a.s. convergence we have to take subsequences. Let $0<\theta<\gamma-1$ and $n_k=\inf\{n : {\e}S_n\ge e^{k/(\log k)^{\theta}} \}$. Let $U_k=S_{n_k}$ and note that the definition and ${\e}X_i\le M$ imply $e^{k/(\log k)^{\theta}} \le {\e}U_k< e^{k/{\log k}^{\theta}}+M$. Replacing $n$ by $n_k$ in (\ref{lemma}) we get
\begin{equation*}
      \p(|U_k - {\e}U_k|>\sigma{\e}U_k)\le \tilde{C} \sigma^{-2} k^{-1}(\log k)^{\theta-\gamma}
\end{equation*}
for some constant $\tilde{C}>0$.
So $\sum_{k=1}^\infty  \p(|U_k - {\e}U_k|>\sigma{\e}U_k)<\infty$, and the Borel-Cantelli lemma (BC1) implies  $\p(|U_k - {\e}U_k|>\sigma{\e}U_k\,\text{i.o.})=0$. Since $\sigma$ is arbitrary, it follows that $U_k/{\e}U_k\to 1$ a.s. To get $S_n/{\e}S_n\to 1$ a.s., pick an $\omega$ so that $U_k(\omega)/{\e}U_k\to 1$ and observe that if $n_k\le n<n_{k+1}$ then
\[
      \frac{U_k(\omega)}{{\e}U_{k+1}}\le  \frac{S_n(\omega)}{{\e}S_n}\le  \frac{U_{k+1}(\omega)}{{\e}U_k}.
\]
To show that the terms at the left and right end converge to 1, we rewrite the last inequalities as
\[
  \frac{{\e}U_{k}}{{\e}U_{k+1}}    \frac{U_k(\omega)}{{\e}U_{k}}\le  \frac{S_n(\omega)}{{\e}S_n}\le  \frac{U_{k+1}(\omega)}{{\e}U_{k+1}} \frac{{\e}U_{k+1}}{{\e}U_k}.
\]
From this we see it is enough to show  ${\e}U_{k+1}/{\e}U_k\to 1$. Since
\[
       e^{k/(\log k)^{\theta}} \le {\e}U_k\le  {\e}U_{k+1}\le  e^{(k+1)/(\log (k+1))^{\theta}}+M
\]
we must show $e^{(k+1)/(\log (k+1))^{\theta}}/ e^{k/(\log k)^{\theta}}$ converges to 1. 

We note that  $e^{k/(\log k)^{\theta}}=k^{h(k)}$
where $h\colon (1, \infty) \to (0,\infty),\, h(x)=x/(\log x)^{1+\theta}$. Then we must show
\[
    \frac{(k+1)^{h(k+1)}}{k^{h(k)}}=\left(1+\frac{1}{k}\right)^{h(k)} e^{(h(k+1)-h(k)) \log (k+1)}
\]
converges to 1. Clearly $h(k)=o(k)$ and so the left hand side of product above converges to $e^{0}=1$. We note that 
$0<h'(x)<(\log x)^{-1-\theta}$ and $h''(x)<0$ (for all sufficiently large $x$), so the  right hand side of product above is 
$\le \exp(\log(k+1)/(\log k)^{1+\theta})$ which also converges to 1. 
\end{proof}

\subsection{Proofs of 1.2}
\begin{proof}[Proof of Theorem \ref{return1}]\text{ }\\
\emph{The simple random walk.}
We treat this particular case separately because its proof is elementary. 
Then we say how it extends easily to lattice random walks by using a local central limit theorem.

We will use Stirling's formula
\[  
        n!=\sqrt{2\pi n} \left(\frac{n}{e}\right)^n \left(1+O\left(\frac{1}{n}\right)\right)
\]
where $O(n^{-1})>0$.
In this particular case let $n_i$ be a sequence of \emph{even} positive integers and let $E_i$ denote the event $S_{n_i}=a\sqrt{n_i}$ whose precise meaning is $S_{n_i}=2[a\sqrt{n_i}/2]$ (this eliminates probability zero of $E_i$). 

\emph{(a)} We have
\[
    \p(E_i)=\binom{n_i}{\left(n_i+2[a\sqrt{n_i}/2]\right)/2} 2^{-n_i}\sim  \sqrt{\frac{2}{\pi}} \,n_i^{-1/2} \, e^{-a^2/2}
\]     
by Stirling's formula, and the result follows from (BC1).

\emph{(b)} First of all, we notice the first condition on the sequence $n_i$ implies
$(\sqrt{n_j}-\sqrt{n_i})^{-1}\le \max\{1,4A^{-1}\}$ for every $j>i$. This will be important because when we do approximation to the lattice  sometimes we will have to use $a \pm 2(\sqrt{n_j}-\sqrt{n_i})^{-1}$ instead of $a$. Let $\tilde{a}=|a|+2 \max\{1,4A^{-1}\}$.

Also, the first condition on the sequence $n_i$ implies there exists $A_1>0$ such that $n_i\ge A_1 i^2$ for every $i$, and so
\begin{equation}\label{2condseq}
     A_2  (\log n)^{1/2} (\log \log n)^{3/2} (\log \log \log n)^{\gamma}  \le\sum_{i=1}^n n_i^{-1/2}\le A_2^{-1} \log n
\end{equation}
for some constant $A_2>0$.

We want to apply Theorem \ref{cheb} to the random variables $\tilde{X}_i=1_{E_i}$. Let $\tilde{S}_n=\sum_{i=i}^n \tilde{X}_i$. We have ${\e}\tilde{S}_n\sim(2/\pi)^{1/2}  e^{-a^2/2}\sum_{i=1}^n n_i^{-1/2} \to \infty$. To verify condition (\ref{rest}) in Theorem \ref{slln} we have $\sum_{i=1}^n \mathrm{var}( \tilde{X}_i)\le{\e}\tilde{S}_n$ and, for $i<j$,
\begin{align*}
     \p(E_i\cap E_j)&=\p(S_{n_i}=a\sqrt{n_i})\p(S_{n_j-n_i}=2[a\sqrt{n_j}/2]-2[a\sqrt{n_i}/2])\\
     &= \p(S_{n_i}=a\sqrt{n_i})\p(S_{n_j}=a\sqrt{n_j}) \left(\frac{n_j}{n_j-n_i}\right)^{1/2} R\\
     &=\p(E_i)\p(E_j) \left(\frac{n_j}{n_j-n_i}\right)^{1/2} R
\end{align*}
where, for $n_j\ge 2\tilde{a}^6$,
\[
    R= e^{\frac{a^2\sqrt{n_i}}{\sqrt{n_j}+\sqrt{n_i}}} \left(1+O\left(\frac{1}{n_j-n_i}+\frac{\tilde{a}^3}{\sqrt{n_j}}\right)\right)
\]
(for $a=0$ this holds with $\tilde{a}=0$)
is obtained using Stirling's formula and 
\[ 
      (1+k/m)^m\le e^k \le (1+k/m)^m (1+k^2/m),
\]
\[
     (1+k/m)^m \le (1-k/m)^{-m}\le (1+k/m)^m  (1-k^2/m)^{-1},
\]
where $k\ge0,\,m> k^2$ and, for the second line of inequalities, we also assume $m\ge 1,\,m> k$.
To estimate
\begin{equation}\label{erdos11}
\Bigl|\sum_{1\le i<j\le n}  \left(\p(E_i\cap E_j)-\p(E_i)\p(E_j)\right)\Bigr|
\end{equation}
we separate the sum in two cases. Let $\sqrt{\nu}_n=(\log {\e}\tilde{S}_n) (\log \log {\e}\tilde{S}_n)^{\gamma}$. In all cases we restrict to $n_i\ge 2\tilde{a}^6$ and $\nu_n\ge a^4$.
Here $C_1, C_2,...$ denote appropriate absolute constants (which do not depend on $a$).

\emph{Case 1}: $n_j>\nu_n n_i$\\
Then we see that 
\[  
    \left(\frac{n_j}{n_j-n_i}\right)^{1/2}\le 1+\frac{1}{\nu_n} \;\text{ and }\; R= 1+O\left(\frac{a^2}{\sqrt{\nu_n}}+\frac{1}{n_j-n_i}+\frac{\tilde{a}^3}{\sqrt{n_j}}\right).
\]
Since $\sum_{i=1}^\infty \p(E_i)/\sqrt{n_i}\le C_1<\infty$, this implies (\ref{erdos11}) is less than some constant $C_2$ times
\[
 \frac{({\e}\tilde{S}_n)^2}{\nu_n} + a^2\frac{({\e}\tilde{S}_n)^2}{\sqrt{\nu_n}} + \frac{{\e}\tilde{S}_n}{\nu_n} + 
 \tilde{a}^3 \frac{{\e}\tilde{S}_n}{\sqrt{\nu_n}}.
\]

\emph{Case 2}: $n_j\le\nu_n n_i$\\
Clearly (\ref{erdos11}) is less than
\begin{align}
    &C_3 e^{a^2/2} \sum_{i,j}\p(E_i)\p(E_j) \left(\frac{n_j}{n_j-n_i}\right)^{1/2} \notag \\
    &\le C_4 \sum_{i=1}^n \p(E_i) \sum_j (n_j-n_i)^{-1/2}\label{erdos2}
\end{align}
Given $i$, let $N$ be the number of $j$'s satisfying $n_i\le n_j \le \nu_n n_i$. Then $n_{i+N}\le \nu_n n_i$ and $N+i\le C_5\nu_n i (\log i)^{1/2}$. The first condition on the sequence $n_i$ implies $n_j-n_i\ge C_6 (j^2-i^2)$ for all $i<j$, so applying Cauchy-Schwarz inequality we get  
\begin{align*}
    \sum_{j=i+1}^N (n_j-n_i)^{-1/2}\le C_7 \left(\sum_{j=i+1}^{N+i} (j+i)^{-1} \right)^{1/2} 
    \left(\sum_{j=i+1}^{N+i} (j-i)^{-1}\right)^{1/2} 
\end{align*}
which is less than $C_8 (\log\log n)^{1/2} (\log n)^{1/2}$. Then (\ref{erdos2}) is less than
\[
C_9 (\log \log n)^{1/2} (\log n)^{1/2} {\e}\tilde{S}_n=O\left(({\e}\tilde{S}_n)^2/\sqrt{\nu_n}\right), 
\]
where we have used (\ref{2condseq}).

Applying Theorem \ref{cheb} we get
\[
\frac{\sum_{i=1}^n 1_{\{S_{n_i}=a\sqrt{n_i}\}}}{\sum_{i=1}^n n_i^{-1/2}} \to  \sqrt{\frac{2}{\pi}}\, e^{-a^2/2} \quad\text{a.s.}
\]
\text{ }\\
\emph{Lattice random walks.} 
Since ${\e}|X_i|^3<\infty$, we have the following local central limit theorem with rates (see \cite{8}):
\begin{equation}\label{lclt1}
   \p(S_n=a\sigma\sqrt{n})=\frac{h}{\sqrt{2\pi n}\,\sigma}\, e^{-a^2/2} \left(1+ O\left(\frac{1}{\sqrt{n}}\right)\right).
\end{equation}
Here, $C_1, C_2, ...$ denote appropriate absolute constants that might depend (continuously) on $a, \sigma, h$ and on the distribution of $X_i$ (this includes the constants in $O(\cdot)$). 
Let $E_i$ denote the event $S_{n_i}=a\sigma\sqrt{n_i}$, $\tilde{X}_i=1_{E_i}$ and $\tilde{S}_n=\sum_{i=i}^n \tilde{X}_i$. By (\ref{lclt1}) we have ${\e}\tilde{S}_n\sim h(\sqrt{2\pi}\sigma)^{-1}  e^{-a^2/2}\sum_{i=1}^n n_i^{-1/2} \to \infty$, $\sum_{i=1}^n \mathrm{var}( \tilde{X}_i)\le{\e}\tilde{S}_n$ and, for $i<j$,
\begin{align*}
     \p(E_i\cap E_j)&=\p(S_{n_i}=a\sigma\sqrt{n_i})\p(S_{n_j-n_i}=a\sigma(\sqrt{n_j}-\sqrt{n_i})+O(1))\\
     &= \p(S_{n_i}=a\sqrt{n_i})\p(S_{n_j}=a\sqrt{n_j}) \left(\frac{n_j}{n_j-n_i}\right)^{1/2} R\\
     &=\p(E_i)\p(E_j) \left(\frac{n_j}{n_j-n_i}\right)^{1/2} R
\end{align*}
where, for $n_j\ge C_1$,
\[
    R= e^{\frac{a^2\sqrt{n_i}}{\sqrt{n_j}+\sqrt{n_i}}} \left(1+O\left(\frac{1}{\sqrt{n_j-n_i}}\right)\right).
\]
Notice that we also used here the second condition on the sequence $n_i$ because we need $a+O(1)(\sqrt{n_j}-\sqrt{n_i})^{-1}$ to 
be uniformly bounded in order to apply (\ref{lclt1}).
Now the rest of the proof is similar to the simple random walk.

The uniform convergence in probability is an immediate consequence of (\ref{lemma}) and the fact that $C$ can be chosen uniform for $a\in [-N,N]$. If we use more restrict sequences $n_i$ then we can improve this uniform convergence. For example,
if, moreover, $n_i\le Ai^2(\log i)^{\alpha}$ for some $0\le\alpha<1$, then, for every $N>0$ and $\gamma>1$ there exists $C>0$ such that, for every $\epsilon>0$,
\[
\sup_{a\in [-N,N]} \p\left(|\Delta_n^a|>\epsilon\right)\le C \epsilon^{-2}(\log \log n)^{-\gamma}
\]  
(just use $\sqrt{\nu_n}=(\log {\e}\tilde{S}_n)^\gamma$).

\end{proof}

\begin{proof}[Proof of Theorem \ref{return3}]
Since ${\e}|X_i|^3<\infty$, we have the following local central limit theorem with rates (see \cite{8}):
\begin{equation}\label{lclt2}
   \p(S_n=a\sigma\sqrt{n})=\frac{1}{\sqrt{2\pi }}\, e^{-a^2/2} \left(1+ O\left(\frac{1}{\sqrt{n}}\right)\right).
\end{equation}
Here, $C_1, C_2, ...$ denote appropriate absolute constants that might depend (continuously) on $a$ and on the density of $X_i$ (this includes the constants in $O(\cdot)$). 
Let $E_i$ denote the event $S_{n_i}=a\sigma\sqrt{n_i}$ and consider the random variables $\tilde{X_i}=1_{E_i}$ and
$\tilde{S}_n=\sum_{i=i}^n \tilde{X}_i$. As before, we want to apply Theorem \ref{cheb} to the random variables $\tilde{X}_i$.
Clearly (\ref{lclt2}) implies ${\e}\tilde{S}_n\sim  n (\sqrt{2\pi})^{-1} e^{-a^2/2}$. Also $\sum_{i=1}^n \mathrm{var}( \tilde{X}_i)\le{\e}\tilde{S}_n$ and, for $i<j$,
\begin{align*}
     \p(E_i\cap E_j)&=\p(S_{n_i}=a\sigma\sqrt{n_i})\p(S_{n_j-n_i}=a\sigma(\sqrt{n_j}-\sqrt{n_i}))\\
     &= \p(S_{n_i}=a\sqrt{n_i})\p(S_{n_j}=a\sqrt{n_j}) R\\
     &=\p(E_i)\p(E_j) R
\end{align*}
where, for $n_j\ge C_1$,
\[
    R= e^{\frac{a^2\sqrt{n_i}}{\sqrt{n_j}+\sqrt{n_i}}} \left(1+O\left(\frac{1}{\sqrt{n_j-n_i}}\right)\right).
\]
To estimate
\begin{equation}\label{erdos11b}
\Bigl|\sum_{1\le i<j\le n}  \left(\p(E_i\cap E_j)-\p(E_i)\p(E_j)\right)\Bigr|
\end{equation}
we separate the sum in two cases. Let $\sqrt{\nu_n}=(\log {\e}\tilde{S}_n) (\log \log {\e}\tilde{S}_n)^{\alpha-1}$.

\emph{Case 1}: $n_j>\nu_n n_i$\\
Then $R=1+O(\nu_n^{-1/2})$ and (\ref{erdos11b}) is $O(({\e}\tilde{S}_n)^2/\sqrt{\nu_n})$.

\emph{Case 2}: $n_j\le\nu_n n_i$\\
Given $i$, let $N$ be the number of $j>i$ satisfying $n_j \le \nu_n n_i$. Then $n_{i+N}\le \nu_n n_i$ and, using the hypothesis on the sequence $n_i$, we get 
\[
   \sum_{k=i}^{N+i-1} (\log i) (\log \log i)^{\alpha}/i\le C_2 \log \nu_n.
\] 
Then, simple calculus shows that, for all sufficiently large $n$,
\begin{align*}
    N&\le \exp\left(\left( (\log n)^{2} +C_3\log \nu_n/(\log \log i)^{\alpha}\right)^{1/2}\right)-n \\
    &\le C_4 n \left((\log n) (\log\log n)^{\alpha-1}\right)^{-1}.
\end{align*}
Then (\ref{erdos11b}) is less than $C_5 n^{2} \left((\log n) (\log\log n)^{\alpha-1}\right)^{-1}=O\left( ({\e}\tilde{S}_n)^2/\sqrt{\nu_n} \right)$.

The conclusion follows from applying Theorem \ref{cheb}.

In the special case $a=0$,  we have $R=1+O\left((n_j-n_i)^{-1/2}\right)$, for $n_j\ge C_1$. Then in \emph{Case 2} (where the hypothesis on $n_i$ was used) we can use $n_i=i$ to get (\ref{erdos11b}) less than
\[
      C_6 \sum_{i=1}^n \sum_{j=i+1}^{\nu_n i} (j-i)^{-1/2}\le C_7 \sqrt{\nu_n} n^{3/2}<({\e}\tilde{S}_n)^2/\sqrt{\nu_n}.
\]

The uniform convergence in probability is an immediate consequence of (\ref{lemma}) and the fact that $C$ can be chosen uniform for $a\in [-N,N]$. If we use more restrict sequences $n_i$ then we can improve this uniform convergence. For example,
if $n_{i+1}/n_i\ge 1+A(\log i)^{\alpha}/i$ for some $\alpha>1$, then, for every $N>0$ and $1<\gamma<\alpha$ there exists $C>0$ such that, for every $\epsilon>0$,
\[
\sup_{a\in [-N,N]} \p\left(|\Delta_n^a|>\epsilon\right)\le C \epsilon^{-2}(\log n)^{-\gamma}
\]  
(just use $\sqrt{\nu_n}=(\log {\e}\tilde{S}_n)^\gamma$).
\end{proof}

\subsection{Proofs of 1.3}

\begin{proof}[Proof of Theorem \ref{teorprinc}]
We use some notation of Probability. Given two measurable functions $f,g\colon X \to \mathbb{R}$ we denote (whenever it makes sense)
\begin{align*}
 \mu(f)&=\int f\,d\mu,\quad \mathrm{var}(f)=\mu(f^2)-\mu(f)^2,\quad \mathrm{cov}(f,g)=\mu(fg)-\mu(f)\mu(g).
\end{align*}
Given $f_1,...,f_n\colon X \to \mathbb{R}$ measurable functions, if $\mu(f_i^2)<\infty$ then (see \cite{7})
\begin{equation}\label{cov}
  \mathrm{var}(f_1+\cdots f_n)=\mathrm{var}(f_1)+\cdots+\mathrm{var}(f_n) + 2\sum_{1\le i<j\le n}\mathrm{cov}(f_i,f_j).
\end{equation}

Fix $\theta>0$ such that $\gamma>1+(2(1+\theta)-\delta)\beta/(2+\delta)$. For each $n$ let $\tilde{f_n}$ be a Lipschitz function such that $\tilde{f_n}(x)=1$ if $x\in B_n$, $\tilde{f_n}(x)=0$ if $d(x,B_n)>(n(\log n)^{1+\theta})^{-1/\delta}$, $0\le \tilde{f_n}\le 1$ and $\|\tilde{f_n}\|_\mathrm{Lip}\le (n(\log n)^{1+\theta})^{1/\delta}$. Let $f_n=\tilde{f}_n\circ T^n$ and $S_n=\sum_{i=0}^{n-1} f_i$. Note that $\mu(S_n)=\sum_{i=0}^{n-1} \mu(B_i) + O(1)$ and for $\mu$ a.e. $x$, $f_n(x)=1_{B_n}(T^n x)$ except for finitely many $n$ by the Borel-Cantelli lemma (BC1), since $\mu(x : f_n(x)\ne 1_{B_n}(T^n x))=\mu(x : r_n<d(T^n x, p_n)<r_n+(n(\log n)^{1+\theta})^{-1/\delta})< (n(\log n)^{1+\theta})^{-1}$ by assumption (A). So it is enough to prove that $S_n/\mu(S_n)\to 1$ $\mu$ a.e., which we will do by using Theorem \ref{cheb}.

Since $0\le f_i\le 1$, we get that $\mathrm{var}(f_i)\le \mu(f_i^2)\le\mu(f_i)$ and (\ref{cov}) implies
\[
  \mathrm{var}(S_n)\le \mu(S_n)+2 \Biggl(\; \underbrace{\sum_{\substack{0\le i<j\le n-1 \\ j-i\le \nu(n)}} \mathrm{cov}(f_i,f_j)}_I +   \underbrace{\sum_{\substack{0\le i<j\le n-1\\ j-i>\nu(n)}} \mathrm{cov}(f_i,f_j)}_{II} \;\Biggr),
\]
where $\nu(n)=\mu(S_n)(\log n)^{-1} (\log \log n)^{-\rho}$, for some $\rho>1$.
We easily bound $I$ by
\[
   I\le \sum_{j=1}^{n-1} \sum_{i=j-[\nu(n)]}^{j-1} \mu(f_j)\le\mu(S_n) \nu(n).
\]
We use decay of correlations to bound $II$ by
\begin{align*}
   II&\le \sum_{\substack{0\le i<j\le n-1\\ j-i>\nu(n)}} c(j-i) \|\tilde{f_i}\|_\mathrm{Lip} \|\tilde{f_j}\|_\mathrm{Lip}\\
   &\le C (n (\log n)^{1+\theta})^{1/\delta} \sum_{i=0}^{n-2} (i (\log i)^{1+\theta})^{1/\delta} \sum_{j=i+[\nu(n)]+1}^{n-1} \frac{1}{(j-i)^\alpha}\\
  &\le \frac{2C}{\alpha-1} \frac{(n (\log n)^{1+\theta})^{2/\delta} n}{\nu(n)^{\alpha-1}}.
\end{align*}
Using the hypothesis on the growth of $\mu(S_n)$ and the definition of $\beta, \gamma$ and $\theta$, we get 
\begin{equation*}
      \mathrm{var}(S_n)= O(\mu(S_n)^2 (\log n)^{-1} (\log \log n)^{-\rho}).
\end{equation*}
Then we satisfy Theorem \ref{cheb}' hyphoteses and so $S_n/\mu(S_n)\to 1$ $\mu$ a.e.

By (\ref{lemma}), there exists $C_1>0$ such that, for every $\epsilon>0$, 
\[
\mu(| S_n/\mu(S_n)-1|>\epsilon)\le C_1 \epsilon^{-2} (\log n)^{-1}(\log\log n)^{-\rho}.
\]
Also by the proof of (BC1),
\[
    \mu(x : f_i(x)\ne 1_{B_i}(T^i x) \text{ for some } i\ge n)\le \sum_{i=n}^\infty \left(i(\log i)^{1+\theta}\right)^{-1}\le C_\theta (\log n)^{-\theta},
\]
for some $C_\theta>0$. Then, using  $\mu(S_n)=\sum_{i=0}^{n-1} \mu(B_i) + O(1)$ and simple inequalities we get
\[
    \mu(|\Delta_n-1|>\epsilon)\le C_2 (\log n)^{-1} (\log \log n)^{-\rho} + C_\theta (\log n)^{-\theta},
\]
where $C_2>0$ depends on $\rho,\theta,\epsilon$. As in the previous section, we can get better `large deviation' results if we increase the growth of $\mu(S_n)$.
\end{proof}

\begin{proof}[Proof of Theorem \ref{teorsec}]
Let $S_n$ be as in the proof of Theorem \ref{teorprinc} (with $\theta=1$). First we prove that
\begin{align}\label{foda}
   (\log n)^\beta \le \frac{2^{\gamma+2}\beta\mu(S_n)}{(\log \mu(S_n))(\log\log \mu(S_n))^\gamma}
\end{align}
for all sufficiently large $n$. Let $\mu(S_n)=\rho(n) (\log n)^\beta$. By hypothesis (and $\mu(S_n)=\sum_{i=0}^{n-1} \mu(B_i) + O(1)$) we get $\rho(n)\ge \frac{1}{2} (\log \log n) (\log\log\log n)^{\gamma}$. Since $x(\log x)^{-1} (\log\log x)^{-\gamma}$ is an increasing function, in order to prove (\ref{foda}) we may assume $\rho(n)= \frac{1}{2} (\log \log n) (\log \log \log n)^\gamma$. Then
\[
   \log \mu(S_n)\le 2\beta \log \log n,\quad (\log \log \mu(S_n))^\gamma \le 2^\gamma (\log\log\log n)^\gamma
\]
which implies (\ref{foda}). 

Now we follow the proof of Theorem \ref{teorprinc} but with $\nu(n)=A(\log n)^{\beta}$, where $\alpha A^{\beta^{-1}}>2/\delta+1$. Then we get
\[
   I\le \mu(S_n) \nu(n)\le \frac{2^{\gamma+2}\beta A \mu(S_n)^2}{(\log \mu(S_n))(\log\log \mu(S_n))^\gamma}
\]
where we have used (\ref{foda}). Also
\[
   II \le C (n (\log n)^{2})^{1/\delta} \sum_{i=0}^{n-2} (i (\log i)^{2})^{1/\delta} \sum_{j=i+[\nu(n)]+1}^{n-1} 
   e^{-\alpha (j-i)^{\beta^{-1}}}
\]
and  $\sum_{k=N}^\infty e^{-\alpha k^{\beta^{-1}}}\le \tilde{C} e^{-\alpha N^{\beta^{-1}}} N^{[\beta]/\beta}$ for some $\tilde{C}>0$, so
\[
II=O\left((n (\log n)^{2})^{2/\delta} n  (\log n)^{\beta} e^{-\alpha \nu(n)^{\beta^{-1}}}\right).
\]
Since $e^{-\alpha \nu(n)^{\beta^{-1}}}=n^{-\alpha A^{\beta^{-1}} }$, the definition of $A$ implies $II\to 0$. Then
\[ 
    \mathrm{var}(S_n)= O\left( \frac{\mu(S_n)^2}{(\log \mu(S_n))(\log\log \mu(S_n))^\gamma}   \right)
\]
and we can apply Theorem \ref{cheb} to get $S_n/\mu(S_n)\to 1$ $\mu$ a.e.

The proof of the `large deviation' result is similar to the one given in the Proof of Theorem \ref{teorprinc}.
\end{proof}

\subsection{Proofs of 1.4}
\begin{proof}[Proof of Theorem \ref{quant}]\text{ }
\begin{enumerate}
\item[(a)]
Note that hypotheses imply $\alpha=\bar{d}_\mu(x_0)>0$.
If we define $C(r)=\mu(B(x_0,r)/r^\alpha$, then $\limsup_{r\to0} \log C(r) / \log r=0$. By hypotheses , there are $r_i\to0$ such that
\begin{equation}\label{holder}
   C(r_i) r_i^\alpha= \mu(B(x_0,r_i)=i^{-\beta}(\log i)^{\gamma}\le C r_i^{\delta(x_0)}.
\end{equation}
Also $r_i^{-1}\ge i^{\frac{\beta}{\alpha}}/ \kappa(i)$ where
\[
 \kappa(i)= \max \left\{1,\, C(r_i)^{-1/\alpha} (\log i)^{\gamma/ \alpha} \right\}.
\]
By (\ref{holder}) we get $\log r_i^{-1}/ \log i \le \beta \delta(x_0)^{-1} +o(1)$ and so $\log \kappa(i)/ \log i \to 0$.
The first part of the result follows by applying Corollary \ref{corol1}.

Now assume $\underline{\Theta}^\alpha_\mu(x_0)>0$. Also assume $\underline{\Theta}^\alpha_\mu(x_0)<\infty$ (the other case is similar). Then, given $0<\epsilon<\underline{\Theta}^\alpha_\mu(x_0)$, there exists $r_0>0$ such that for every $0<r<r_0$ we have $(\underline{\Theta}^\alpha_\mu(x_0)-\epsilon)\le\mu(B(x_0,r)/ r^{\alpha} \le (\underline{\Theta}^\alpha_\mu(x_0)+\epsilon)$. Then we set $r_i^\alpha=i^{-\beta}(\log i)^{\gamma}$ and apply
Corollary \ref{corol1}.

\item[(b)]
The proof is similar to the proof of (a) with the obvious modifications and using Corollary \ref{cor2}.
\end{enumerate}
\end{proof}

\textbf{Acknowledgments:} The author was partially supported by Funda\c{c}\~ao para a Ci\^encia e a
Tecnologia through the project "Randomness in Deterministic Dynamical
Systems and Applications" (PTDC/MAT/105448/2008), and by PRONEX-Dynamical Systems,  CNPQ/FAPERJ.

\end{document}